\documentclass[10pt, final]{amsart}
\usepackage[english]{babel}
\usepackage{latexsym}
\usepackage{amssymb}
\usepackage{amscd}
\usepackage[cp850]{inputenc}
\usepackage[mathscr]{eucal}

\newcommand{\adef}{\begin{defn}}
\newcommand{\zdef}{\end{defn}}
\newtheorem{defn}{Definition}

\newcommand {\N}{\mathbb N}

\newcommand {\R}{\mathbb R}

\newcommand {\C}{\mathbb C}
\newcommand {\U}{\mathscr U}

\newtheorem{theorem}{Theorem}[section] 
\newtheorem{lemma}[theorem]{Lemma}     
\newtheorem{corollary}[theorem]{Corollary}
\newtheorem{proposition}[theorem]{Proposition}

\newtheorem{defi} [theorem]{Definition}
\newtheorem{fact}[theorem]{Fact}

\newcommand\dist{\mathop{\mathrm{dist}}\nolimits}

\newcommand{\aproof}{\begin{proof}}
\newcommand{\zproof}{\end{proof}}

\title[On UFO Banach spaces]{On Uniformly finitely extensible Banach spaces}

\author{Jes\'us M. F. Castillo, Valentin Ferenczi and Yolanda Moreno}

\thanks{46B03, 46B07, 46B08, 46M18, 46B25, 46B42.}

\thanks{This research has been supported in part by project
MTM2010-20190-C02-01 and the program Junta de Extremadura GR10113
IV Plan Regional I+D+i, Ayudas a Grupos de Investigaci\'on, and
also by CNPQ projeto 455687/2011-0 }
\begin{document}

\maketitle

\begin{abstract} We continue the study of Uniformly Finitely
Extensible Banach spaces (in short, UFO) initiated in
Moreno-Plichko, \emph{On automorphic Banach spaces}, Israel J.
Math. 169 (2009) 29--45 and Castillo-Plichko, \emph{Banach spaces
in various positions.} J. Funct. Anal. 259 (2010) 2098-2138. We
show that they have the Uniform Approximation Property of Pe\l
czy\'nski and Rosenthal and are compactly extensible. We will also
consider their connection with the automorphic space problem of
Lindenstrauss and Rosenthal --do there exist automorphic spaces
other than $c_0(I)$ and $\ell_2(I)$?-- showing that a space all
whose subspaces are UFO must be automorphic when it is
Hereditarily Indecomposable (HI), and a Hilbert space when it is
either locally minimal or isomorphic to its square. We will
finally show that most HI --among them, the super-reflexive HI
space constructed by Ferenczi-- and asymptotically $\ell_2$ spaces
in the literature cannot be automorphic.
\end{abstract}

\maketitle

\section{Introduction and preliminaries}

For all the unexplained notation and terms see the second half of this section.
In this paper we continue the study of Uniformly Finitely
Extensible Banach spaces (in short, UFO) initiated in
\cite{castplic,moreplic} and their connection with the automorphic
space problem \cite{castmoreisr,lindrose,avimore}. Following \cite{castmoreisr}, a Banach space is said to be
\emph{automorphic} if  every isomorphism between two subspaces
such that the corresponding quotients have the same density
character can be extended to an automorphism of the whole space.
Equivalently, if for every closed subspace $E\subset X$ and every
into isomorphism $\tau:E\to X$ for which $X/E$ and $X/\tau E$ have
the same density character, there is an automorphism $T$ of $X$
such that $T|_E=\tau$. The motivation for such definition is in
the Lindenstrauss-Rosenthal theorem \cite{lindrose} asserting that
$c_0$ is automorphic and the extension presented in
\cite{moreplic} -- for every $\Gamma$ the space $c_0(\Gamma)$ is
automorphic--. This leads to the generalization of the still open
problem set by Lindenstrauss and Rosenthal \cite{lindrose}\\

\noindent \textbf{Automorphic space problem}: Does there exist an
automorphic space different from $c_0(I)$ or
 $\ell_2(I)$?\\

The papers \cite{avimore,moreplic} and \cite{castplic} considered
different aspects of the automorphic space problem. In particular,
the following two groups of notions were isolated:

\adef  A couple $(Y, X)$ of Banach spaces is said to be ({\it compactly}) {\it extensible} if for
every subspace $E \subset Y$ every (compact) operator $\tau: E \to X$ can be
extended to an operator $T: Y \to X$. If there is a $\lambda>0$
such that some extension exists
verifying $\|T\|\leq \lambda\|\tau\|$ then we will
say that $(Y, X)$ is $\lambda$-(compactly) extensible. The space $X$ is said to be
({\it compactly}) {\it extensible} if $(X, X)$ is ({\it compactly}) extensible and {\it uniformly} ({\it compactly}) {\it extensible} if it is $\lambda$-(compactly) extensible for some $\lambda$.\zdef

It is not known whether there exist separable extensible spaces
different from $c_0$ and $\ell_2$. Neither it is known if an
extensible space must be uniformly extensible, although some
partial results have already been obtained in \cite{moreplic} and
\cite{castplic}; precisely, that an extensible space isomorphic to
its square is uniformly extensible.

\adef A couple $(Y,X)$ of Banach spaces is said to be a
$\lambda$-{\it UFO pair} if for every finite dimensional subspace
$E$ of $Y$ and every linear operator $\tau:E\to X$, there exists a
linear continuous extension $T:Y\to X$ with $\|T\|\leq
\lambda\|\tau\|$. A couple $(Y,X)$ of Banach spaces is said to be
an {\it UFO pair} if it is a $\lambda$-UFO pair for some
$\lambda$. A Banach space $X$ is said to be {\it Uniformly
Finitely Extensible} (an UFO, in short) if $(X,X)$ is an UFO pair.
It is said to be a $\lambda$-UFO if $(X,X)$ is a $\lambda$-UFO
pair.\zdef

It is clear that every
$\mathcal L_{\infty, \lambda}$-space is a $\lambda$-UFO. Recall
that a subspace $Y$ of a Banach space $X$ is said to be locally
complemented if $Y^{**} = Y^{\perp\perp}$ is complemented in
$X^{**}$. Some acquaintance with ultraproduct theory will be
required: Let $I$ be a set, $\U$ be an ultrafilter on $I$, and
$(X_i)$ a family of Banach spaces. Then $ \ell_\infty(X_i)$
endowed with the supremum norm, is a Banach space, and $
c_0^\U(X_i)= \{(x_i) \in \ell_\infty(X_i) : \lim_{\U(i)}
\|x_i\|=0\} $ is a closed subspace of $\ell_\infty(X_i)$. The
ultraproduct of the spaces $X_i$ following $\U$ is defined as the
quotient
$$
[X_i]_\U = {\ell_\infty(X_i)}/{c_0^\U(X_i)}.
$$
We denote by $[(x_i)]$ the element of $[X_i]_\U$ which has the
family $(x_i)$ as a representative. It is not difficult to show
that $ \|[(x_i)]\| = \lim_{\U(i)} \|x_i\|. $ In the case $X_i = X$
for all $i$, we denote the ultraproduct by $X_\U$, and call it the
ultrapower of $X$ following $\U$. The following lemma gathers the
basic results on UFO spaces from \cite{moreplic} (results (i, ii))
 and \cite{castplic} (results (iii, iv).

\begin{lemma}\label{minga} $\;$ \begin{description}
\item[i)] Every compactly extensible space is an UFO. \item[ii)] Every
$\lambda$-UFO that is $\mu$-complemented in its bidual is
$\lambda\mu$-extensible. \item[iii)] A locally complemented subspace of
an UFO is an UFO. \item[iv)] Ultrapowers of $\lambda$-UFO are
$\lambda$-UFO; consequently, biduals of $\lambda$-UFO are
$\lambda$-extensible.
\end{description}
\end{lemma}

The spaces $Y$ for which $(Y,\ell_p)$ is an UFO pair were
investigated in \cite{CasazzaNi} under the name $M_p$-{\it
spaces}, and Maurey's extension theorem (see e.g. \cite{Z}) can be
reformulated in this language as: \emph{Each type 2 space is
$M_2$}. Related to this is the so-called \emph{Maurey extension
property} (in short MEP): a Banach space $X$ has MEP if every
operator $t: E\to \ell_2$ from any subspace $E$ of $X$ admits an
extension to $T:X\to \ell_2$. The equivalence between $M_2$ and
MEP should be known but we have been unable to find any explicit
reference. The result follows from the following generalization(s)
of Lemma 1.1 (ii):

\begin{lemma}\label{extbidual}
If the pair $(Y, X)$ is $\lambda$-UFO then $(Y, X^{**})$ is
$\lambda$-extensible. Therefore, if $X$ is complemented in its
bidual, the pair $(Y, X)$ is an UFO if and only if $(Y, X)$ is
extensible.
\end{lemma}

\noindent \emph{Sketch of proof}. Assume that $(Y, X)$ is
$\lambda$-UFO. Let $E$ be an infinite dimensional subspace of $Y$
and $\phi: E\to X$ an operator. For each finite dimensional
subspace $F\subset E$ let $\phi_F$ be the restriction of $\phi$ to
$F$ and then let $\Phi_F$ be its extension to $Y$ verifying
$\|\Phi_F\|\leq \lambda\|\phi_F\|$, which exists by hypothesis.
Let FIN the partially ordered set of all finite dimensional
subspaces of $E$ and let $\mathscr U$ be an ultrafilter refining
the Fr\'echet filter. Let $\psi: Y\to X^{**}$ be defined as
$$\psi(x) =\mathrm{weak}^*-\lim_{\mathscr U(F)} \Phi_F(x)).$$
The proof that if $(Y, X)$ is extensible then $(Y, X)$ is an UFO
pair is similar to that of Lemma 1.1 (i).\hfill $\square$

\begin{corollary} Properties $M_2$ and $MEP$ are equivalent.
\end{corollary}

The contents of the paper are as follows. \\

In section 2 we show that UFO spaces have the Uniform
Approximation Property, as a consequence of the following general
principle: If there is a constant $C$ so that all $X$-valued norm
one finite-rank operators defined on subspaces of $X$ admit
extensions to $X$ with norm at most $C$ then for every
$\varepsilon> 0$ they admit finite-rank extensions with norm at
most $(C+\varepsilon)$. As a consequence we show that when $X$ is
an UFO, $X$-valued compact operators defined on its subspaces can
be uniformly extended to the whole space, and the extension
operator can even be chosen to be compact. This solves in the
affirmative the question posed in \cite{moreplic} of whether UFO
and compactly extensible spaces coincide. The corresponding
question for all operators (i.e., whether extensible implies
uniformly extensible) is open. The following diagram displays the basic implications:

\begin{eqnarray*} \mathrm{automorphic} \Rightarrow \mathrm{unif.} \; \mathrm{extensible} \Rightarrow \mathrm{extensible} \Rightarrow \mathrm{compactly} \; \mathrm{extensible}\\
\Leftrightarrow\\
\mathrm{unif.} \; \mathrm{compactly} \; \mathrm{extensible}\\
\Leftrightarrow\\
\mathrm{UFO.}
\end{eqnarray*}

It was obtained in \cite{castplic} --see also Theorem \ref{dicho} below-- the dichotomy principle asserting that an UFO must be either i) an
$\mathcal L_\infty$-space or ii) a B-convex near-Hilbert space (i.e.,
\cite{castplic}, a space $X$ such that $p(X)=q(X)=2$ where $p(X)=\sup\{p:X\;{\rm
is\,of\,type}\;p\}\;\; \mathrm{and}\;\; q(X)=\inf\{q:X\;{\rm
is\,of\,cotype}\;q\}$) with the Maurey Extension Property, called B-UFO from now on. In section 3 we focus on classifying B-UFO in the presence of some
additional properties; mainly: to be isomorphic to its square or
to be Hereditarily Indecomposable (HI from now on), which somehow
are properties at the two ends of the spectrum. For instance, we
will show that a space all whose subspaces are UFO must be
automorphic when it is HI, and a Hilbert space when it is either
isomorphic to its square or locally minimal. Among other stability
properties of the class of UFO spaces, we prove that $X\oplus
\ell_2$ is B-UFO when $X$ is B-UFO, however it still remains
unsolved the question whether the product of two B-UFO spaces must
be be B-UFO. Actually, a positive answer to that question would
imply that every hereditarily UFO is Hilbert. Although the
quotient of two UFO does not have to be UFO, we prove that when
$Y$
is locally complemented in an UFO space $X$ then $X/Y$ is UFO. The following table could help the reader:\\

 \begin{tabular}{|l|l|l|}

    \hline

    Type       & $\mathcal L_\infty$                             & $B$-convex                                 \\

    \hline

    UFO                    & basis    &   near Hilbert, MEP, UAP\\

   UFO +  HI                   & exist    & ? \\

 UFO minimal & $c_0$    &   superreflexive \\

    \hline

    HUFO & do not exist  & exist \\

 HUFO locally minimal & do not exist  & Hilbert \\

 HUFO isomorphic to square  & do not exist & Hilbert\\
    \hline
  \end{tabular}

  \vspace{0,5cm}


Section 4 examines possible counterexamples to the automorphic
space problem. Indeed, while the papers \cite{moreplic} (resp.
\cite{avimore}) showed that most of the currently known Banach
spaces (resp. $C(K)$-spaces) cannot be automorphic, a work
continued in \cite{castplic}, in this paper we turn our attention
to HI or asymptotically $\ell_2$ spaces to show that most of the
currently known examples cannot be automorphic. Special attention
is paid to the only HI space which is candidate to be automorphic: the
super-reflexive HI space $\mathcal F$ constructed by the second
author in \cite{F1}, for which we show it is not UFO since it is
not near-Hilbert (see Prop. \ref{type}). The concluding
Section 5 contains the technical results required to prove this last assertion.\\

\section{Compactly extensible spaces}

It is known that extensible spaces are UFO (it follows from Lemma
1.1 (i)) and that UFO spaces are not necessarily extensible (say,
$C[0,1]$). The question of whether UFO spaces must be compactly
extensible was posed in \cite{moreplic} and will be affirmatively
solved in Proposition \ref{compext} below. Moreover, we will show
that compactly extensible spaces are uniformly compactly
extensible. This is remarkable since, as we have already
mentioned, it is an open question whether an extensible space must
be uniformly extensible. The following general principle is the
key:

\begin{lemma}\label{genprop} If $(Y, X)$ is a $\lambda$-UFO pair then,
for every $\varepsilon > 0$, every $X$-valued operator $t$ defined
on a finite-dimensional subspace of $Y$ admits a finite-rank
extension to the whole of $Y$ with norm at most $(\lambda
+\varepsilon)\|t\|$.
\end{lemma}
\begin{proof}
Let $F$ be a finite-dimensional subspace of $Y$ and let $t: F\to X$ be an
operator. Take $(y_n)$ to be a finite $\varepsilon$-net of the unit sphere of $F$,
with $\varepsilon<1$, then pick norm one functionals
$(f_n)$ so that $f_n(y_n)=\|y_n\|$, and form the
finite-codimensional subspace of $Y$

$$H = \bigcap_n \ker f_n.$$

It is clear that $F\cap H=0$: Indeed, if all $f_n$ vanish on some norm one
element $y\in F$ then take $y_k$ in the unit sphere of $F$ with $
\|y-y_k\| \leq \varepsilon$ and thus $\|y_k\|= f_k(y_k)=
f_k(y_k-y) + f_k(y) \leq \varepsilon$, which is impossible. Since
$F$ is finite-dimensional, $F+H$ is closed so $F+H=F\oplus H$.
Given $y$ of norm $1$ in $F$ and $h \in H$, and if $y_k$ is
defined as above, then
$$\|y+h\| \geq f_k(y+h)=f_k(y)=f_k(y-y_k)+f_k(y_k) \geq 1-\varepsilon.$$
This means that the natural projection $p:F\oplus H \to F $
actually has norm at most $(1-\varepsilon)^{-1}$. Then $tp: F
\oplus H\to X$ is a finite-rank operator with norm at most
$(1-\varepsilon)^{-1}\|t\|$. By Lemma \ref{extbidual}, $tp$ admits
an extension $T:Y\to X^{**}$, with norm
$(1-\varepsilon)^{-1}\lambda\|t\|$, which moreover has
finite-dimensional range since $T_{|_{H}} = 0$. So, by the
principle of local reflexivity,  there is an operator $Q: T(Y)\to
X$ with norm at most $(1-\varepsilon)^{-1}$ so that $Q(u)= u$ for
all $u\in T(Y)\cap X$. The operator $QT: Y\to X$ has finite range
and $QT(f) = Qtp(f) = Qt(f) = t(f)$ for all $f \in F$. Moreover
$\|QT\| \leq (1-\varepsilon)^{-2}\lambda\|t\|$.
\end{proof}

Recall that a Banach space $X$ is said to have the \emph{Bounded
Approximation Property} (BAP in short) if for some $\lambda$ and
each finite dimensional subspace $F\subset X$ there is a finite
rank operator $T:X\to X$ such that $\|T\|\leq \lambda$ and
$T(y)=y$ for each $y\in F$ (in which case it is said to have to $\lambda$-BAP). The corresponding notion in local
theory is the \emph{Uniform Approximation Property} (UAP in short)
introduced by Pe\l czy\'nski and Rosenthal \cite{pelcrose} by
asking the existence of a function $f:\N\to \N$ so that  the
choice above can be performed verifying $ \mathrm{rank} T \leq
f(\dim F)$. It is easy to see that $X$ has the UAP if and only if
every ultrapower of $X$ has the BAP.

\begin{proposition}\label{UAP} If $X$ is a $\lambda$-UFO then, for any $\varepsilon>0$, it has the ($\lambda+\varepsilon$)-UAP.
\end{proposition}
\begin{proof} Lemma \ref{genprop} immediately implies that a $\lambda$-UFO must have the $(\lambda+\varepsilon)$-BAP for all $
\varepsilon>0$. Since ultrapowers of $\lambda$-UFO are
$\lambda$-UFO \cite{castplic}, the result is clear.\end{proof}

\noindent \textbf{Remarks.} The previous proof provides an
estimate for the function $f:\N\to\N$ that defines the UAP. From
the proof it follows that if $v: \N\times \R^+\to \N$ is the
function for which $v(n,\varepsilon)$ is the infimum of the $N$
such that every $n$-dimensional subspace of an UFO space $X$
admits an $\varepsilon$-net for its unit sphere with $N$ points,
then it is possible to check the UAP property in $X$ with $f(n) =
v(n,\varepsilon)$ for any choice of $\varepsilon<1$. Volume
estimates indicate that the unit ball of a real space of dimension
$n$ cannot be covered by less than $(\frac{1}{\varepsilon})^n$
balls of radius $\varepsilon$, and a similar estimate holds for
$v(n,\epsilon)$. This is essentially optimal, since actually
$v(n,\varepsilon) \leq (1+2/\varepsilon)^n$ (see \cite{PiVol},
Lemma 4.10).
Therefore if a real space is UFO, then the UAP property may be verified with, for example, $f(n) = 4^n$, which is clearly a bad estimate.  Johnson and Pisier \cite{johnpisi} proved that a Banach space is a weak Hilbert space if and only if $f(n) = O(n).$\\

We now refine Lemma \ref{genprop} in order to show compact
extensibility and UFO are equivalent.






\begin{lemma}\label{casoBAP}
Let $X$ be a space with BAP. The following statements are equivalent:
\begin{enumerate}
\item The pair $(Y, X)$ is UFO. \item There exists $C \geq 1$ such
that for every subspace $E\subset Y$, every finite-rank operator
$t: E\to X$ admits a finite-rank extension with $\|T\|\leq C
\|t\|$.
\item The pair $(Y, X)$ is compactly extensible. \item There
exists $C \geq 1$ such that for every subspace $E\subset Y$, every
compact operator $t: E\to X$ admits a compact extension with
$\|T\|\leq C \|t\|$.
\end{enumerate}
\end{lemma}
\begin{proof} That (iv) implies (iii) and (ii) implies (i) are
obvious. And that (iii) implies (i) is Lemma 1.1 (i) and was
proved in \cite{moreplic} without the need of asking $X$ to
have the BAP.\\

(i) implies (ii): we prove that if  $(Y, X)$ is $\lambda$-UFO and
$X$ has the $\mu$-BAP, then (ii) holds with $C=\lambda
\mu+\varepsilon$, for all $\epsilon>0$. Let $t: E\to X$ be a
finite-rank operator from a subspace $E \subset Y$ and let $F=tE$.
By Lemma \ref{extbidual}, there exists an extension $T: Y\to
X^{**}$ of $t$ with $\|T\|\leq \lambda \|t\|$. Now, since $X$ has
$\mu$-BAP, there is a finite range operator $\tau: X\to X$ such
that $\tau(u) = u$ for every $u\in F$ and $\|\tau\|\leq \mu$.
Consider the bi-adjoint operator $\tau^{**}$ and take the finite
dimensional subspace $\tau^{**}T(Y)\subset X^{**}$. By the
principle of local reflexivity \cite{lindtzaf}, there exists an
operator $Q: \tau^{**}T(Y) \to X$ of norm at most $1+\varepsilon$
such that $Q(x) = x$ for every $x\in \tau^{**}T(Y)\cap X$. The
operator $Q\tau^{**}T : Y\to X$ has finite range and for every
$e\in E$, $Q\tau^{**}T(e) = Q\tau t(e) = \tau (t e) = t e$.
Moreover $\|Q\tau^{**}T\| \leq (1+\varepsilon) \mu\lambda\|t\|$.\\

(ii) implies (iv): we prove that if $X$ has the AP and (2) holds
with $C$, then the pair $(Y,X)$ is $(C+\varepsilon)$-compactly
extensible for any $\varepsilon>0$, with compact extension.
Consider any subspace $E\subset Y$ and a compact operator $t: E
\to X$, since $X$ has the AP, $t$ can be uniformly approximated by
finite rank operators $t_n: E \to X$, that is $t = \|\cdot\|-\lim
t_n$, and we may assume $\|t_n\|=\|t\|$ for all $n$. Let
$\varepsilon >0$, without lost of generality (passing to a
subsequence if necessary) we may assume that $\|t_1 - t\|\leq
\varepsilon$ and   $\|t_n-t_{n-1}\| \leq \varepsilon_n$, for
$n\geq 2$, with $(\varepsilon_n)$ a given sequence of positive
numbers such that $\sum_{n = 2}^{+\infty}\varepsilon_n \leq
\varepsilon\|t\|/C$. Consider for every $n\geq 2$ finite-rank
extensions $T_n: Y\to X$ of $\tau_n = t_{n}-t_{n-1}$ with
$\|T_n\|\leq C  \|\tau_n\|$ and let $T_1$ be an extension of $t_1$
such that $\|T_1\|\leq C  \|t_1\|$. The operator $T = \sum T_n$ is
a well defined compact operator which extends $t$ and such that
$\|T\|\leq (C+\varepsilon)\|t\|$. \end{proof}

\begin{proposition}\label{compext}
The following statements are equivalent:
\begin{enumerate}
\item $X$ is UFO. \item $X$ is compactly extensible. \item $X$ is
uniformly compactly extensible.
\end{enumerate}
Moreover the extension operators in (ii) and (iii) may be chosen
compact.\end{proposition}

When $Y$ is a fixed subspace of the Banach space $X$, and $E$ is a
Banach space it is a direct consequence from the open mapping
theorem that if all (compact) operators $Y\to E$ can be extended
to operators $X\to E$ then they can be uniformly extended. Also,
$\mathcal L_\infty$-spaces have the property when acting as target
spaces $E$ that all $E$-valued compact operators can be
(uniformly) extended \cite{lindmemo} and admit compact extensions.

\section{B-convex UFO}
The following dichotomy result was obtained
in \cite{castplic}:

\begin{theorem}[\textbf{Dichotomy principle}]\label{dicho} An UFO Banach space is either an $\mathcal
L_\infty$-space or a $B$-convex near-Hilbert space with MEP.
\end{theorem}

This dichotomy provides an affirmative answer to a question of
Galego \cite{gale}: Is an automorphic subspace of $c_0(\Gamma)$
isomorphic to some $c_0(I)$? Indeed, any infinite dimensional
closed subspace of $c_0(\Gamma)$ contains $c_0$, hence every UFO
subspace of $c_0(\Gamma)$ must be an $\mathcal L_\infty$-space;
and $\mathcal L_\infty$-subspaces of $c_0(\Gamma)$ are isomorphic
to $c_0(I)$ (cf. \cite{gkl}). We will now focus  our attention on
UFO which are not $\mathcal L_\infty$-spaces.

\adef A Banach space is said to be a \emph{B-UFO} if it is a
B-convex UFO. \zdef

Of course the first open question is: Do there exist non-Hilbert
B-UFO spaces? The only candidate currently known seems to be
Tsirelson's 2-convexified space $\mathcal T_2$ \cite{casshu},
which is near-Hilbert, superreflexive and, having type 2, also
enjoys MEP \cite[p.~127]{casshu}. It is not known whether
$\mathcal T_2$ is an UFO (the answer would be negative if
$\mathcal T_2$ would, for instance, contain an uncomplemented copy
of itself since no automorphism of the space can send a complemented subspace into an uncomplemented one. The space $\mathcal T_2$ moreover is \emph{minimal}
--a Banach space $X$ is minimal when every infinite-dimensional
closed subspace contains a copy of $X$--. Now observe that by the
Gowers dichotomy \cite{gowedicho}, every Banach space contains
either an Hereditarily Indecomposable subspace or a subspace with
unconditional basis, and spaces with unconditional basis must be
either reflexive, contain $\ell_1$ or $c_0$
\cite{james}. Thus, if $X$ is a minimal UFO that contains
an HI subspace, $X$ itself must be HI; but HI spaces cannot be
minimal since they are not isomorphic to their proper subspaces.
If, however, $X$ contains either $\ell_1$ or $c_0$, it must be itself be a
subspace of either $\ell_1$ or $c_0$ and an $\mathcal
L_\infty$-space, so it must be $c_0$. All this means that a
minimal UFO must be either $c_0$ or a reflexive near-Hilbert
space with MEP. This can be improved using the local version of minimality. Recall that a Banach space $X$ is said to be \emph{locally minimal} (\cite{fererose}) if there is
some $K \geq 1$ such that every finite-dimensional subspace
$F\subset X$ can be $K$-embedded into any infinite-dimensional
subspace $Y\subset X$. This notion is much weaker than classical
minimality; for instance, every $c_0$-saturated Banach space is
locally minimal.

\begin{proposition} A minimal UFO must be either $c_0$ or a superreflexive near-Hilbert
space with MEP.
\end{proposition}
\begin{proof} By Johnson's local version of James' trichotomy for spaces with unconditional basis
\cite{johnson} we get that a Banach space with local unconditional structure must contain uniformly complemented $\ell_1^n, \ell_\infty^n$ or to be superreflexive. Since minimal implies locally minimal, a minimal UFO containing either $\ell_1^n$ or $\ell_\infty^n$ must be $c_0$; since a reflexive space space cannot be $\mathcal L_\infty$, it must be superreflexive.
\end{proof}

\adef A Banach space is said to be \emph{hereditarily UFO} (a HUFO
in short \cite{castplic}) if each of its closed subspaces is an
UFO.\zdef

In \cite{castplic} it was shown that a HUFO spaces must be
\emph{asymptotically $\ell_2$}, i.e., there is a constant $C>0$
such that for every $n$ there is a finite-codimensional subspace
$X_n$ all whose $n$-dimensional subspaces $G$ verify
$\dist(G,\ell_2^{\dim G})\leq C$. Asymptotically $\ell_2$ spaces
are reflexive (see \cite[p.~220]{PiVol}), and reflexive UFO are
extensible. So, no $\mathcal L_\infty$ space can be HUFO, which
means that every HUFO is (hereditarily) B-UFO.

\begin{proposition}\label{locallyminimal} A locally minimal HUFO is isomorphic to a Hilbert space.
\end{proposition}
\begin{proof} Let $X$ be a locally minimal HUFO space. On one side every HUFO must be asymptotically $\ell_2$; while, on the other
hand, the hypothesis of local minimality  means that there is some
$K \geq 1$ so that for every infinite-dimensional subspace
$Y\subset X$ and every finite-dimensional subspace $F\subset X$
there is a subspace $F_Y\subset Y$ such that $\dist(F, F_Y)\leq
K$. Now, if $X$ is not a Hilbert space, there is a sequence
$(H_n)$ of finite-dimensional subspaces such that $\lim \dist(H_n,
\ell_2^{\dim H_n})= \infty$
---otherwise $X$ would be a subspace of the Hilbert space formed as the ultraproduct
of all the finite-dimensional subspaces of $X$---. So $H_n$ cannot
$K$-embed into the finite-codimensional subspaces $X_m$ of the
definition of asymptotically $\ell_2$ for large $m$.
\end{proof}

In \cite{fererose} a dichotomy is proved opposing local minimality
to {\em tightness with constants} (one formulation of this
property is that a space $X$ with a basis is tight with constants
if no subspace of $X$ is crudely finitely representable into all
tail subspaces of $X$). The proof of Proposition
\ref{locallyminimal} indicates that a HUFO space with a basis is
either tight with constants or Hilbert. One also has:

\begin{proposition}\label{stable} HUFO spaces isomorphic to their square are isomorphic to Hilbert spaces.
\end{proposition}
\begin{proof} This follows from \cite{castplic} where it was
proved that if $X$ contains a subspace of the form $A\oplus A$
with $A\neq \ell_2$ then $X$ contains a non-UFO
subspace.\end{proof}

A basic open question is whether the product of two B-UFO is a
B-UFO. It is well known that the product of two $\mathcal
L_\infty$-spaces is an $\mathcal L_\infty$-space, and it is clear
that $c_0\oplus \ell_2$ cannot be an UFO (it is not $B$-convex or $\mathcal L_\infty$); so the question
above is the only case that has to be elucidated.

\begin{proposition}\label{eledos} If $X$ is B-UFO then $\ell_2\oplus X$ is B-UFO.
\end{proposition}
\begin{proof} Since $X$ is B-convex, it
contains $\ell_2^n$ uniformly complemented and hence, for some
ultrafilter $\mathscr U$, the ultraproduct $X_{\mathscr U}$
contains $\ell_2$ complemented. Hence $X_{\mathscr U} \simeq
\ell_2 \oplus Z \simeq \ell_2 \oplus \ell_2 \oplus Z \simeq \ell_2
\oplus X_{\mathscr U}$ is an B-UFO, as well as its locally
complemented subspace $\ell_2\oplus X$.\end{proof}


\begin{corollary} If for some sequences of scalars $p_n\neq 2$ and  naturals $k_n$ a Banach space
$X$ contains finite dimensional $l_{p_n}^{k_n}$  uniformly
complemented then it is not an UFO.\end{corollary}
\begin{proof} The hypothesis means that $\ell_2\oplus X$ contains
$\ell_2^n\oplus\ell_{p_n}^{k_n}$ uniformly complemented. Using
\cite[Lemma 6.11]{castplic} we get copies of $\ell_{p_n}^{k_n}$
which are not uniformly complemented in $X$. An appeal to
\cite[Thm. 4.4]{moreplic} shows that $X$ is not an UFO.
\end{proof}

This improves the analogous result in \cite[Cor. 4.6]{moreplic}
for fixed $p_n=p \neq 2$. Another intriguing partial result is the
following.

\begin{lemma} If the product of two B-UFO is always an B-UFO then every HUFO is isomorphic to  a Hilbert space.\end{lemma}
\begin{proof} Let $X$ be HUFO; then $X\oplus X$
is UFO. If $X$ is not Hilbert, following the proof  in
\cite{castplic}, $X\oplus X$ must contain a non-UFO subspace with
FDD of the form $\sum (G_k\oplus F_k)$, where $G_k\subset X$,
$F_k\subset X$ and the $F_k$ containing badly complemented
subspaces uniformly isomorphic to $G_k$. The space $\sum G_k
\oplus \sum F_k \subset X\oplus X$ is not UFO by \cite[Th.
4.4]{moreplic}. But since $X$ is HUFO, both $\sum G_k$ and $\sum
F_k$ must be UFO.
\end{proof}

To be an UFO is by no means a 3-space property (see
\cite{castgonz}) because non-trivial twisted Hilbert spaces
--i.e., spaces $Z$ containing an uncomplemented copy of $\ell_2$
for which $Z/\ell_2$ is isomorphic to $\ell_2$-- cannot be B-UFO
since the MEP property would make every copy of $\ell_2$
complemented. Such twisted sum spaces exist: see \cite{castgonz}
for general information and examples. The quotient of two
$\mathcal L_\infty$ spaces is an $\mathcal L_\infty$-space, hence
UFO. Nevertheless, it was shown in \cite{castplic,accgm} that
$\ell_\infty/c_0$ is not extensible, which implies that the
quotient of two extensible spaces does not have to be extensible.
Let us show that the quotient of two UFO is not necessarily an
UFO.
\begin{lemma} The space $\ell_\infty/\ell_2$ is not UFO:
\end{lemma}
\begin{proof} Indeed, since $B$-convexity is a $3$-space property \cite{castgonz},
it cannot be $B$-convex. Let $p:\ell_\infty^{**}\to \ell_\infty$
be a projection through the canonical embedding $ \ell_\infty\to
\ell_\infty^{**}$. The commutative diagram
$$\begin{CD}
0@>>>\ell_2 @>>> \ell_\infty @>>> \ell_\infty/\ell_2 @>>> 0 \\
&& @| @AAA @AAuA \\ 0@>>>\ell_2 @>>> \ell_\infty^{**} @>>>
\ell_\infty^{**}/\ell_2 @>>>0
\end{CD}$$
immediately implies $\ell_\infty/\ell_2$ is complemented in its
bidual. Thus, if $\ell_\infty/\ell_2$ was an $\mathcal
L_\infty$-space then it would be injective. Let us show it is not:
Take an embedding $j: \ell_2\to L_1(0,1)$ and consider the
commutative diagram
$$\begin{CD}
0@>>>\ell_2 @>>> \ell_\infty @>>> \ell_\infty/\ell_2 @>>> 0 \\
&& @| @AAA @AAuA \\ 0@>>>\ell_2 @>>> Z @>>> \ell_2  @>>>0.
\end{CD}$$
in which the lower sequence is any nontrivial twisted sum of
$\ell_2$. If there is an operator $U: L_1(0,1) \to
\ell_\infty/\ell_2$ such that $Uj=u$ then one would get a
commutative diagram
$$\begin{CD}
0@>>>\ell_2 @>>> \ell_\infty @>>> \ell_\infty/\ell_2 @>>> 0 \\
&& @| @AAA @AAUA \\
0@>>>\ell_2 @>>> X @>>> L_1(0,1) @>>>0\\
 && @| @AAA @AAjA\\
0 @>>>\ell_2 @>>> Z @>>> \ell_2 @>>>0,
\end{CD}$$which is impossible since the $\ell_2$ subspace in the middle exact sequence must be complemented
by the Lindenstrauss lifting principle \cite{lindl1} while in the
lower sequence it is not.\end{proof} Nevertheless, when the
subspace is locally complemented the quotient must be UFO:

\begin{proposition}\label{X/Yloc} Let $X$ be an UFO and $Y$ a locally complemented subspace of some ultrapower $X_\U$ of $X$.
Then $X_\U/Y$, $Y \oplus X$ and $X\oplus (X_\U/Y)$ are all UFO.
\end{proposition}
\begin{proof} The space $X_\U/Y$ is UFO since it is locally complemented in $(X_\U/Y)^{**}$, which in turn is complemented in
$X_\U^{**}$. The space $Y\oplus X$ is locally complemented in
$X_{\mathscr U}\oplus X_{\mathscr U}$, which is an UFO, while
$X\oplus (X_\U/Y)$ is locally complemented in $X_{\mathscr
U}\oplus X_{\mathscr U}^{**}$, which is also UFO.
\end{proof}


A couple of questions about B-UFO for which we  conjecture an
affirmative answer are: Is every B-UFO  reflexive? Is every B-UFO
isomorphic to its square? It can be observed that if every B-UFO
contain a subspace of the form $A\oplus A$ for
infinite-dimensional $A$ then every HUFO must be Hilbert.\\

Turning our attention to the other end of the spectrum, spaces
which do not contain subspaces isomorphic to its square, we find
the Hereditarily Indecomposable (HI, in short) spaces. Recall that
a Banach space is said to be HI if no subspace admits a
decomposition in the direct product of two infinite dimensional
subspaces. Hereditarily indecomposable UFO exist after the
constructions of HI $\mathcal L_\infty$-spaces obtained by Argyros
and Haydon \cite{argyhay}, and later by Tarbard \cite{tarb}. Let
us show that, in some sense, spaces which are simultaneously HI
and UFO are close to automorphic.

\begin{lemma} For HI spaces, extensible and automorphic are
equivalent notions.
\end{lemma}
\begin{proof} Operators on HI spaces are either strictly singular
of Fredholm with index $0$. The extension of an embedding cannot
be strictly singular, so it is an isomorphism between two
subspaces with the same finite codimension, which means that some
other extension of the embedding is an automorphism of the whole
space.
\end{proof}

Hence, taking into account Lemma 1.1 (ii) one gets:

\begin{proposition} A reflexive space which is HI
and UFO is automorphic.\end{proposition}

In particular, HI spaces which also are HUFO must be automorphic.
A variation in the argument shows:

\begin{proposition}  Assume that  $X$ is an $\mathcal
L_\infty$ space with the property that every operator $Y\to X$
from a subspace $i:Y \to X$ has the form $\lambda i + K$ with $K$
compact. Then $X$ is automorphic.
\end{proposition}
\begin{proof} Since compact operators on an $\mathcal L_\infty$-space can be extended to larger superspaces, the hypothesis
implies that the space is extensible. The hypothesis also yields
that $X$ must be HI.
\end{proof}

These two propositions suggest a way to obtain a counterexample
for the automorphic space problem. We however conjecture that
an HI space cannot be automorphic.\\

\section{Counterexamples}

As we said before, there are not many examples known of either
near-Hilbert spaces or spaces with MEP. For a moment, let us focus
on $B$-convex spaces. A good place to look for them is among weak
Hilbert spaces. Recall that a Banach space $X$ is said to be weak
Hilbert if there exist constants $\delta, C$ such that every
finite dimensional subspace $F$ contains a subspace $G$ with $\dim
G\geq \delta \dim F$ so that $\dist(G, \ell_2^{\dim G})\leq C$. It
is well-known that weak Hilbert spaces are $B$-convex and
near-Hilbert, although it is not known if they must have MEP.
Tsirelson's 2-convexified space $\mathcal T_2$ \cite{casshu} is a
weak Hilbert type 2 space, but we do not know whether it is an UFO.
Nevertheless, since $\mathcal T_2$ is isomorphic to its square we
can apply Proposition \ref{stable} to conclude it is not HUFO.
Since subspaces of weak Hilbert spaces are weak Hilbert, we obtain that
not all weak Hilbert spaces are UFO. This answers the question
left open in \cite[p.2131]{castplic} of whether weak Hilbert
spaces must be UFO. Argyros, Beanland and Raikoftsalis \cite{abr}
have recently constructed a weak Hilbert space $X_{abr}$ with an
unconditional basis in which no disjointly supported subspaces are
isomorphic (such spaces are called tight by support in
\cite{fererose}). Clearly this space does not contain a copy of
$\ell_2$. By the criterion of Casazza used by Gowers in its
solution to Banach's hyperplane problem \cite{Ghyperplanes},
tightness by support implies that $X_{abr}$ is not isomorphic to
its proper subspaces, and in particular is not isomorphic to its
square. It remains open whether this space
is an UFO or even a HUFO.\\

``Less Hilbert" than weak Hilbert spaces are the asymptotically
$\ell_2$ spaces, still to be considered since HUFO spaces are of
this type. See also \cite{anis} for related properties.
Asymptotically $\ell_2$ HI spaces have been constructed by
different people. The space of Deliyanni and Manoussakis \cite{deliman}
cannot be HUFO since it has the property that $c_0$ is finitely
represented in every subspace, so it is locally minimal. Apply now
Proposition \ref{locallyminimal}. We do not know however if this
space or if the asymptotically $\ell_2$ HI space constructed by
Androulakis and Beanland \cite{androbean} are UFO. In
\cite{casjohn} Casazza, Garc\'ia and Johnson construct an
asymptotically $\ell_2$ space without BAP; which, therefore,
cannot be UFO. Actually, the role of the BAP in these UFO affairs
is another point not yet understood. Johnson and Szankowski
introduce in \cite{johnszank} HAPpy spaces as those Banach spaces
all whose subspaces have the approximation property. Szankowski
had already shown in \cite{szan} that HAPpy spaces are
near-Hilbert, while Reinov \cite{rei} showed the existence of a
near-Hilbert non-HAPpy space. This motivates the following question: Is every B-UFO
space HAPpy? Another construction of Johnson and Szankowski in
\cite{johnszank} yields a HAPpy asymptotically $\ell_2$ space with
the property that every subspace is isomorphic to a complemented
subspace. This is a truly wonderful form of not being extensible;
and since the space is reflexive, of not being
UFO.\\

Passing to more general HI spaces, the Argyros-Deliyanni
asymptotically $\ell_1$ space \cite{argdeli} was shown in
\cite{moreplic} not to be UFO. Argyros and Tollias \cite[Thm.
11.7]{argtol} produce an HI space X so that both $X^*$ and
$X^{**}$ are HI and $X^{**}/X = c_0(I)$; they also produce
\cite[Thm. 14.5]{argtol} for every Banach space $Z$ with basis not
containing $\ell_1$ an asymptotically $\ell_1$ HI space $X_Z$ for
which $Z$ is a quotient of $X_Z$. Argyros and Tollias extend they
result \cite[Thm. 14.9]{argtol} to show that for every separable
Banach space $Z$ not containing $\ell_1$ there is an HI space
$X_Z$ so that $Z$ is a quotient of $X_Z$. The space $X_Z$ can be
obtained applying a classical result of Lindenstrauss (see
\cite{lindtzaf}) asserting that every separable $Z$ is a quotient
$E^{**}/E$ of a space $E^{**}$ with basis. Apply the previous
result to $E^{**}$ to obtain an asymptotically $\ell_1$ HI space
$X_Z$ such that $X_Z^{**}/X_Z = E^{**}$ which makes also $Z$ a
quotient of $X_Z$. None of them can be UFO:

\begin{lemma} Asymptotically $\ell_1$ spaces cannot be UFO.
\end{lemma}
\begin{proof} Asymptotically $\ell_1$ spaces contain $\ell_1^n$
uniformly and UFO spaces containing $\ell_1^n$ are $\mathcal
L_\infty$-spaces; which cannot be asymptotically
$\ell_1$.\end{proof}

Passing to $\mathcal L_\infty$-spaces, we show that the
Argyros-Haydon $\mathcal{AH}$ space \cite{argyhay} is not
automorphic.

\begin{proposition}
The space $\mathcal {AH}$ is not extensible; hence it cannot be
automorphic.
\end{proposition}
\begin{proof}
Indeed, each operator in $\mathcal {AH}$ is a sum of scalar and
compact operators, but there is a subspace $Y\subset \mathcal
{AH}$ and an operator $\tau:Y\to Y$ which is not a sum of scalar
and compact operators. This operator cannot be extended onto the whole
space $\mathcal {AH}$.
\end{proof}

Argyros and Raikoftsalis \cite{arra} constructed another separable
$\mathcal L_\infty$ counterexample to the scalar-plus-compact
problem. However it contains $\ell_1$ and

\begin{lemma} No separable Banach space containing $\ell_1$ can be
extensible.
\end{lemma}
\begin{proof} Indeed, the proof of Theorem \ref{dicho} in \cite{castplic} actually shows that
an extensible space containing $\ell_1$ must be separably
injective, and Zippin's theorem \cite{zipp} yields that $c_0$ is
the only separable separably injective space.\end{proof}

Related constructions are those of a different HI $\mathcal
L_\infty$-space of Tarbard \cite{tarb} and that of Argyros,
Freeman, Haydon, Odell, Raikoftsalis, Schlumprecht and
Zisimopoulou, who show in \cite{afhorsz} that every uniformly
convex separable Banach space can be embedded into an $\mathcal
L_\infty$-space with the scalar-plus-compact property. We do not know whether these spaces can be automorphic.\\

On the other side of the dichotomy, essentially the only known
example of a uniformly convex (hence $B$-convex) HI space was
given by the second author in 1997 \cite{F1}. Other examples
include, of course, its subspaces, and also a variation of this
example, all whose subspaces fail the Gordon-Lewis Property, due
to the second author and P. Habala \cite{FH}. From now on the space
of \cite{F1} will be denoted by $\mathcal F$. The space $\mathcal
F$ is defined as the interpolation space in $\theta \in ]0,1[$ of
a family of spaces similar to Gowers-Maurey's space and of a
family of copies of $\ell_q$ for some $q \in ]1,+\infty[$. We
shall prove the following result concerning the type and cotype of
$\mathcal F$. Here $p \in ]1,q[$ is defined by the relation
$1/p=1-\theta+\theta/q$, and we recall that $p(\mathcal
F)=\sup\{t:{\mathcal F}\;{\rm is\,of\,type}\;t\}\;\;
\mathrm{and}\;\; q(\mathcal F)=\inf\{c:{\mathcal F}\;{\rm
is\,of\,cotype}\;c\}$.

\begin{proposition}\label{type} We have the following estimates
\begin{itemize}
\item[(1)] if $q \geq 2$ then
\begin{itemize}
\item [(a)] $\mathcal F$ has type $[1-(\theta/2)]^{-1}$ and $p(F)
\leq p$, \item[(b)]  $\mathcal F$ has cotype $q/\theta$ and
$q(\mathcal F)=q/\theta$ ,
\end{itemize}
\item[(2)] if $q \leq 2$ then
\begin{itemize}
\item[(a)] $\mathcal F$ has type $p$ and $p(\mathcal F)=p$, \item[(b)]
$\mathcal F$ has cotype $2/\theta$ and $q(\mathcal F) \geq q/\theta$.
\end{itemize}
\end{itemize}
In particular $\mathcal F$ is not near-Hilbert and therefore it is
not UFO.
\end{proposition}

It may be interesting to observe that if we choose $q=2$ and
$\theta$ sufficiently close to $1$, then the above estimates imply
that for any $\epsilon>0$, $\mathcal F$ may be chosen to be of
type $2-\epsilon$ and cotype $2+\epsilon$. We leave  the technical
proof of Proposition \ref{type} for the last section of this
paper.

\section{Determination of type and cotype of ${\mathcal F}$}

 This section is devoted to the proof of  Proposition \ref{type}.  Recall that the space $\mathcal F$ is a complex space defined as the
interpolation space in $\theta$ of a family of spaces $X_t, t \in
\R$, on the left of the border of the strip ${\mathcal S}=\{z \in
\C: 0 \leq Re(z) \leq 1\}$, and of a family of copies of the space
$\ell_q$ on the right of the border of $\mathcal S$, based on  the
theory of interpolation of a family of complex norms on $\C^n$
developed  in \cite{5a,5b}. Here $1<q<+\infty$ and $0<\theta<1$.
So it is more adequate to say that in \cite{F1} is produced a
family of uniformly convex and hereditarily indecomposable
examples depending on the parameters $\theta$ and $q$. Each space
$X_t$  is quite similar to the HI Gowers-Maurey space $GM$
\cite{GM}, and this occurs in a uniform way associated to a coding
of analytic functions.  Since those spaces
 satisfy approximate lower $\ell_1$-estimates, it follows that $\mathcal
F$ satisfy approximate lower $\ell_p$-estimates, where $p$ is defined by the classical
interpolation formula $\ell_p \simeq (\ell_1,\ell_q)_{\theta}$,
that is $\frac{1}{p}=(1-\theta)+\theta/q$; this is the main tool
used in \cite{F1} to prove that  $\mathcal F$ is uniformly convex.

On the other hand combined results of  Androulakis - Schlumprecht
\cite{AS} and Kutza\-rova - Lin \cite{KL} imply that the space
$GM$ contains $\ell_\infty^n$'s uniformly. In what follows we
shall prove that this result extends to each space $X_t$, and
furthermore that this happens uniformly in $t$. We shall then
deduce by interpolation methods that the space $\mathcal F$
contains arbitrary long sequences satisfying upper $\ell_r$
estimates, where $\frac{1}{r}=\theta/q$. From this we shall deduce
that $\mathcal F$ cannot have cotype less than $r$, whereas by the
approximate lower $\ell_p$-estimates, it does not have type more
than $p$. Therefore $\mathcal F$ may not be near-Hilbert and
neither may it be UFO.\\

As it is also known from the arguments of \cite{KL} and \cite{AS}
that any subspace of $GM$ contains $\ell_\infty^n$'s uniformly, it
is probable that our proof would also apply to deduce that no
subspace of ${\mathcal F}$ is near-Hilbert and therefore UFO. The
same probably holds for the Ferenczi-Habala space which is
constructed by the same interpolation method as above, using a
variation of the HI space $GM$.\\

We shall call $(e_n)$ the standard vector basis of $c_{00}$, the
space of eventually null sequences of scalars. We use the standard
notation about successive vectors in $c_{00}$. In particular the
{\em support} of a  vector $x=\sum_i x_i e_i$ in $c_{00}$ is ${\rm
supp\ }x=\{i \in \N: x_i \neq 0\}$ and the {\em range} of $x$ is
the interval of integers ${\rm ran\ }x=[\min({\rm supp\
}x),\max({\rm supp\ }x)]$, or $\emptyset$ if $x=0$. Also, if $x=\sum_i x_i e_i \in c_{00}$ and $E=[m,n]$ is an interval of integers, then
$Ex$ denotes the vector $\sum_{i=m}^n x_i e_i$. We recall that
Schlumprecht's space  $S$ \cite{S} is defined by the implicit
equation on $c_{00}$:

$$\|x\|_S=\|x\|_{\infty} \vee \sup_{n \geq 2, E_1<\cdots<E_n}\frac{1}{f(n)}\sum_{k=1}^n \|E_k x\|_S,$$

where $f(x)=\log_2(x+1)$ and $E_1,\ldots,E_n$ are successive
intervals of integers. Therefore every finitely supported vector
in $S$ is normed either by the sup norm, or by a functional of the
form $\frac{1}{f(j)}\sum_{s=1}^j x_s^*: x_1^*<\cdots<x_j^*$ where
$j \geq 2$ and each $x_s^*$ belongs to the unit ball $B(S^*)$ of
$S^*$. For $l \geq 2$, we define the equivalent norm $\|\cdot\|_l$
on $S$ by
$$\|x\|_l=\sup_{E_1<\cdots<E_l}\frac{1}{f(l)}\sum_{j=1}^l \|E_j x\|_S.$$
This norm corresponds to the supremum of the actions of
functionals of the form $\frac{1}{l}\sum_{s=1}^j x_s^*:
x_1^*<\cdots<x_l^*$ where each $x_i^*$ belongs to $B(S^*)$. It
will be useful to observe that if $x$ is a single vector of the
unit vector basis of $S$, then $\|x\|_l = \frac{1}{f(l)}$.

\

In Gowers-Maurey's type constructions a third term associated to
the action of so-called "special functionals" is added in the
implicit equation. We proceed to see how this is done  in the case
of $\mathcal F$. For the rest of this paper, $q \in ]1,+\infty[$
and $\theta \in ]0,1[$ are fixed, and $p \in ]1,q[$ is given by
the formula
$$1/p=1-\theta+\theta/q.$$ As was already mentioned, the space $\mathcal F$ is defined as the
interpolation space of two vertical lines of spaces, a line of
spaces $X_t$, $t \in \R$, on the left side of the strip ${\mathcal
S}$  and a line of  copies of $\ell_q$ on the right side of it. In
\cite{5a,5b} the spaces $X_t$ need only be defined for $t$ in a
set of measure $1$, and for technical reasons in the construction
of $\mathcal F$ in \cite{F1} the $X_t$'s are of interest only for
almost every $t$ real. So in what follows, $t$ will always be
taken in some set $S_{0}^{\infty}$ of measure $1$ which is defined
in \cite{F1}.\\

Our interest here is on the spaces $X_t$ and their norm
$\|\cdot\|_t$. As written in \cite[p. 214]{F1} we have the
following implicit equation for $x \in c_{00}$:
$$\|x\|_t=\|x\|_{\infty} \vee \sup_{n \geq 2, E_1<\cdots<E_n}
\frac{1}{f(n)} \sum_{k=1}^n \|E_k x\|_t \vee \sup_{G\ special,
E}|EG(it)(x)|.$$ The first two terms are similar to Schlumprecht's
definition. Now {\em special analytic functions} $G$ on the strip
${\mathcal S}$ have a very specific form, and special functionals
in $X_t^*$, for $t \in S_0^{\infty}$,   are produced by taking the
value $G(it)$ in $it$ of  special analytic functions. A special
analytic function is of the following form:
$$G=\frac{1}{\sqrt{f(k)}^{1-z} k^{z-z/q}}(G_1+\ldots+G_k),$$
where each $G_j$ is of the form
$$G_j=\frac{1}{f(n_j)^{1-z} n_j^{z-z/q}}(G_{j,1}+\ldots+G_{j,n_j}),$$
with each $G_{j,m}$ an analytic  function on the strip such that
$G_{j,m}(it)$ belongs to the unit ball of $X_t^*$ for each $t$ and
$G_{j,m}(1+it)$ belongs to the unit ball of $(\ell_q)^*$ for each
$t$. In the above, all analytic functions are finitely supported, i.e. for any analytic function $G$, there exists an interval of integers $E$ such that ${\rm supp}\ G(z) \subset E$ for all $z$, and therefore it makes sense to talk about successive analytic functions; moreover in the formulas above it is assumed that $G_1<\cdots<G_k$, respectively $G_{j,1}<\cdots<G_{j,n_j}$. Furthermore it is required that $n_1=j_{2k}$ and
$n_j=\sigma(G_1,\ldots,G_{j-1})$ for $j \geq 2$, where
$j_1,j_2,\ldots$ is the increasing  enumeration of a sufficiently
lacunary subset $J$ of $\N$ and $\sigma$ is some injection of some
set of finite sequences of analytic functions into $J$.
 We refer to \cite{F1} for more details. This means that for $t \in S_{0}^{\infty}$, any special functional
$z^*$ in $X_t^*$, obtained by the formula $z^*=G(it)$ will have
the form
$$z^*=\lambda \frac{1}{\sqrt{f(k)}} (z_1^*+\ldots+z_k^*),$$
where $|\lambda|=1$ and each $z_j^* \in B(X_t^*)$ has the form
$$z_j^*=\frac{1}{f(n_j)}(z_{j,1}^*+\ldots+z_{j,n_j}^*),$$
with each $z_{j,m}^*$ in the unit ball of $X_t^*$, with
$n_1=j_{2k}$, and for $j \geq 2$,
$n_j=\sigma(G_1,\ldots,G_{j-1})$, where $G_1,\ldots,G_k$ is a
sequence of analytic functions on ${\mathcal S}$ such that
$G_j(it)=z_j^*$ for each $j$. Observe that if two initial segments
of two sequences $z_1^*,\ldots,z_i^*$ and
$z_1^{*\prime},\ldots,z_j^{*\prime}$ defining two special
functionals are different, then the associated initial segments of
sequences of analytic functions must be different as well, and
therefore by the injectivity of $\sigma$ the integers $n_{i+1}$
and $n_{j+1}^{\prime}$ associated to $z_{i+1}^*$ and
$z_{j+1}^{*\prime}$ will be different. So the part of classical
Gowers-Maurey procedure using coding may be applied in spaces
$X_t$. G. Androulakis and Th. Schlumprecht proved that the
spreading model of the unit vector basis of $GM$ is isometric to
the unit vector basis of $S$ \cite{AS}. Similarly, we shall prove
that the unit vector basis of $S$ is "uniformly" the spreading
model of the unit vector basis of any $X_t$. We shall base the
proof on a technical lemma, which is essentially Lemma 3.3 from
\cite{AS}, and that we state now but whose proof will be
postponed.\\

For an interval $I \subset \N$ we define
$$J(I)=\{\sigma(G_1,\ldots,G_n): n \in \N,\ G_1<\cdots<G_n,\ \min I \leq \max {\rm supp\ } G_n <\max I\}.$$

For $z^* \in c_{00}$, we define $$J(z^*)=J({\rm ran\ }z^*).$$

It will be useful to observe that whenever ${\rm ran\ }z^* \subset
{\rm ran\ }w^*$, then $J(z^*) \subset J(w^*)$.

\begin{lemma}\label{lemma}
Let $t \in S_{0}^{\infty}$. There exists a norming subset $B_t$ of
the unit ball of $X_t^*$ such that for any $z^* \in B_t$ there
exists $T_0(z^*) \in B(S^*)$, and a family $(T_j(z^*))_{j \in
J(z^*)} \subset B(S^*)$ such that
\begin{enumerate}
\item for $j \in \{0\} \cup J(z^*)$, ${\rm ran\ }(T_j(z^*))
\subset {\rm ran\ }(z^*)$, \item for $j \in J(z^*)$,
$$T_j(z^*) \in aco\{\frac{1}{f(j)}\sum_{s=1}^j x_s^*: x_1^*<\cdots<x_j^* \in B(S^*)\},$$
where "aco" denotes the absolute convex hull, \item
$$z^*=T_0(z^*)+\sum_{j \in J(z^*)}T_j(z^*).$$
\end{enumerate}
\end{lemma}

The proof of Lemma \ref{lemma} is given at the end of the article.

\begin{proposition}\label{spreadingmodel} Let $\epsilon>0$ and $k \in \N$. Then there exists $N \in \N$ such that for any $t \in S_{0}^{\infty}$ and for any
$N \leq n_1 < \cdots < n_k$,  for any scalars $(\lambda_i)_i$,
$$\|\sum_{i=1}^k \lambda_i e_i\|_S
\leq \|\sum_{i=1}^k \lambda_i e_{n_i}\|_t \leq
(1+\epsilon)\|\sum_{i=1}^k \lambda_i e_i\|_S.$$
\end{proposition}

\begin{proof} It is based on the similar lemma of \cite{AS} relating $GM$ to $S$.
The left-hand side inequality is always true by the respective
definitions of $X_t$ and $S$, so we concentrate on the right-hand
side. Fix $\epsilon>0$. We may assume that $\max_i |\lambda_i|=1$.
Since $J$ is lacunary enough we can find $M$ sufficiently large,
such that
$$k  \sum_{l \in J, l \geq M}\frac{1}{f(l)} <\epsilon.$$
Since $\sigma$ is injective, there exists an $N$ such that the
condition $N \leq \max{\rm supp}(G_n)$ guarantees that
$\sigma(G_1,\ldots,G_n)$ is at least $M$. In other words, whenever
$l \in J([N,+\infty))$, then $l \geq M$, and therefore
$$k \max_i |\lambda_i| \sum_{I \in J([N,+\infty))}\frac{1}{f(l)}<\epsilon.\  (1)$$
Thus if $N \leq n_1 < \ldots <n_k$, and $t$ is arbitrary in
$S_0^{\infty}$, then by Lemma \ref{lemma},

$$\|\sum_{i=1}^k \lambda_i e_{n_i}\|_t \leq \|\sum_{i=1}^k \lambda_i e_i\|_S+\sum_{l \in J([N,\infty))}\|\sum_{i=1}^k \lambda_i e_{n_i}\|_l$$
$$\leq \|\sum_{i=1}^k \lambda_i e_i\|_S+k\max_i|\lambda_i| \sum_{I \in J([N,+\infty))}\frac{1}{f(l)} \leq \|\sum_{i=1}^k \lambda_i e_i\|_S+\epsilon.$$ Together with the fact that the basis of $S$ is bimonotone, this concludes the proof.
 \end{proof}

\

\begin{defi}
A vector in $c_{00}$ of the form $$\frac{f(m)}{m}\sum_{i \in
K}e_i,$$ where $m=|K|$, will be said to be an  $S$-normalized
constant coefficients vector.
\end{defi}

 It is observed in \cite{KL} that such a vector has norm $1$ in Schlumprecht's space. In each $X_t$ any sequence of successive vectors satisfies the inequality
$$\frac{1}{f(n)}\sum_{i=1}^n \|x_i\|_t \leq \|\sum_{i=1}^n x_i\|_t,$$
from which we deduce immediately that any $S$-normalized constant
coefficients vector must have norm at least $1$ in each $X_t$.

\begin{proposition}\label{KL} Let $n \in \N$, $\epsilon>0$. There exists a sequence of vectors $u_1,\ldots,u_n$ in $c_{00}$ such that
\begin{itemize}
\item[(a)] the support of the $u_j$'s are pairwise disjoint,
\item[(b)] each $u_j$ is an $S$-normalized constant coefficients
vector, \item[(c)]  for any $t \in S_0^\infty$, for each $j$, $1
\leq \|u_j\|_t \leq  1+\epsilon$, \item[(d)] for any $t \in
S_0^\infty$, any $a_1,\ldots,a_n$ of modulus $1$, $\|a_1
u_1+\ldots+a_n u_n\|_t \leq 1+\epsilon$.
\end{itemize}
\end{proposition}

\begin{proof} D. Kutzarova and P. K. Lin \cite{KL} have proved that for any $n \geq 1$ there exists a sequence of
$S$-normalized constant coefficient vectors  $u_1,\ldots,u_n$ in
$c_{00}$ which are disjointly supported and such that
$\|u_1+\ldots+u_n\|_S \leq 1+\epsilon$.  Recall that the basis of
$S$ is $1$-unconditional, so as well $\|a_1 u_1+\ldots+a_n u_n\|_S
\leq 1+\epsilon$ whenever the $a_j$'s have modulus $1$. Since the
basis of $S$ is $1$-subsymmetric, by Proposition
\ref{spreadingmodel} we may assume that $u_1,\ldots,u_n$ were
taken far enough on the basis to guarantee that $\|u_j\|_t \leq
(1+\epsilon)^2$ and  $\|a_1 u_1+\ldots+a_n u_n\|_t \leq
(1+\epsilon)^2$ for any $t \in S_0^{\infty}$ and any choice of
$a_1,\ldots,a_n$. Since $\epsilon$ was arbitrary this proves the
result.
\end{proof}

\begin{defi}
A vector in $c_{00}$  of the form
$$\frac{f(m)^{1-\theta}}{m^{1/p}}\sum_{i \in K}e_i,$$ where
$m=|K|$,  will be said to be an $S_{\theta}$-normalized constant
coefficients vector.
\end{defi}

 Such a vector would have norm $1$ in the $\theta$-interpolation space of $S$ and $\ell_q$; see \cite{K} where such spaces are studied. For our purposes we shall only use the following fact:

\begin{fact} Any $S_{\theta}$-normalized vector $x \in c_{00}$ satisfies $\|x\|_F \geq 1$.
\end{fact}

\begin{proof} By \cite{F1} Proposition 1, for any successive vectors $x_1<\cdots<x_m$  in $c_{00}$,
$$\frac{1}{f(m)^{1-\theta}}\big(\sum_{k=1}^m \|x_k\|_F^p\big)^{1/p}
\leq \|\sum_{k=1}^m x_k\|_F \leq \big(\sum_{k=1}^m
\|x_k\|_F^p\big)^{1/p}. $$ In the case of an
$S_{\theta}$-normalized constant coefficients vector the left-hand
side gives the result.
\end{proof}

\begin{proposition}\label{upper} Let $n \in \N$, $\epsilon>0$. There exists a sequence of vectors $v_1,\ldots,v_n$ in $c_{00}$ such that
\begin{itemize}
\item[(a)] the supports of the $v_j$'s are pairwise disjoint,
\item[(b)] each $v_j$ is an $S_{\theta}$-normalized constant
coefficients vector, \item[(c)] for each $j$, $1 \leq \|v_j\|_F
\leq 1+\epsilon$, \item[(d)] for any $a_1,\ldots,a_n$ of modulus
$1$, $\|a_1 v_1+\ldots + a_n v_n\|_F \leq (1+\epsilon)
n^{\theta/q}$.
\end{itemize}
\end{proposition}

\begin{proof} Let $u_1,\ldots,u_n$ be given by Proposition \ref{KL}. Write each $u_j$ in the form
$$u_j=\frac{f(m_j)}{m_j}\sum_{i \in M_j}e_i,$$
where $m_j=|M_j|$. Consider for each $j$ the analytic function
$F_j$ defined on ${\mathcal S}$ by
$$F_j(z)=\frac{f(m_j)^{1-z}}{m_j^{1-z+z/q}} \sum_{i \in M_j}e_i.$$
Let $v_j=F_j(\theta)$ for each $j$. Observe that
$$v_j=\frac{f(m_j)^{1-\theta}}{m_j^{1/p}}\sum_{i \in M_j}e_i,$$
and therefore $v_j$ is an $S_{\theta}$-normalized constant
coefficients vector and has norm at least $1$. Also the $v_j$'s
are disjointly supported. Let $a_1,\ldots,a_n$ be an arbitrary
sequence of complex numbers of modulus $1$. Denote $y=a_1
v_1+\ldots+a_n v_n$ and $F=a_1 F_1+\ldots+a_n F_n$. Since
$\mathcal F$ is an analytic bounded function on $S$ satisfying
$F(\theta)=y$,  it belongs to the set ${\mathcal A}_{\theta}(y)$
of analytic functions defined at the beginning of \cite{F1} 1.2.
By Lemma 1 of \cite{F1} the following formula holds for $x \in
c_{00}$:
$$\|x\|_F=\inf_{G \in {\mathcal A}_{\theta}(x)}\Big(\int_R\|G(it)\|_t d\mu_0(t)\Big)^{1-\theta}
\Big(\int_R\|G(1+it)\|_q d\mu_1(t)\Big)^{\theta} \ \ \ \ \\ \ \ \
\ \ (2),$$ where $\mu_0$ and $\mu_1$ are some probability measures
on $\R$ whose definitions may be found in \cite{F1}. Therefore
$$\|y\|_F \leq \Big(\int_R\|F(it)\|_t d\mu_0(t)\Big)^{1-\theta}
\Big(\int_R\|F(1+it)\|_q d\mu_1(t)\Big)^{\theta}  \ \ \ \ \ \ \ \
\ \  (3),$$

Now for any $t$ in $S_0^{\infty}$,
$$F_j(it)=\frac{f(m_j)}{m_j} \Big(\frac{f(m_j)}{m_j^{1-1/q}}\Big)^{-it} (\sum_{i \in M_j}e_i)=
a_{j,t} u_j,$$ where $a_{j,t}$ has modulus $1$. Therefore
$$F(it)=\sum_{j=1}^n a_j a_{j,t} u_j,$$
and by Proposition \cite{KL}, $$\|F(it)\|_t \leq 1+\epsilon \ \ \
\ \ \ \ \ \ (4).$$  On the $\ell_q$-side we compute that
$$F_j(1+it)=\frac{1}{m_j^{1/q}}\Big(\frac{f(m_j)}{m_j^{1-1/q}}\Big)^{-it} \sum_{i \in M_j}e_i,$$
therefore $\|F_j(1+it)\|_q=1$, and since the vectors $F_j(1+it)$
are disjointly supported,
$$\|F(1+it)\|_q=n^{1/q}. \ \ \ \ \ \ \ \ \ \ \ \ (5)$$

Combining (3)(4) and (5),

$$\|y\|_F \leq (1+\epsilon)^{1-\theta} n^{\theta/q} \leq (1+\epsilon) n^{\theta/q}.$$

Applying (2) to each $v_j$ and considering the estimates obtained
for $F_j(it)$ and $F_j(1+it)$, we also obtain that

$$\|v_j\|_F \leq (1+\epsilon)^{1-\theta} 1^{\theta} \leq 1+\epsilon.$$
\end{proof}

We pass to the proof of Proposition \ref{type}\\

\noindent \emph{Proof of Proposition \ref{type}} First we observe
that since
 $p<q/\theta$ it follows from these estimates that $\mathcal F$ can never be near-Hilbert. To prove the estimates, first note that by \cite{F1} Proposition
1, for any successive sequence of normalized vectors
$x_1,\ldots,x_n$ in $\mathcal F$ we have that
$$\|x_1+\ldots+x_n\| \geq \frac{n^{1/p}}{f(n)^{1-\theta}}.$$
It follows that if $\mathcal F$ has type $t$ then $n^{1/p} \leq M
f(n)^{1-\theta} n^{1/t}$ for some constant $M$, and since
$\mathcal F$ is logarithmic, that $t \leq p$. Therefore $p(F) \leq
p$. On the other hand from Proposition \ref{upper}, we see
immediately that if $\mathcal F$ has cotype $c$ then $c$ must be
at least $q/\theta$, so $q(F) \geq q/\theta$. If $q \geq 2$ then
it follows from the inequality appearing in \cite{F1} Proposition
3 that the modulus of convexity in $\mathcal F$ has power type
$q/\theta$; from which by results of Figiel and Pisier \cite[Thm.
1.e.16]{lindtzaf}  $\mathcal F$ has cotype $q/\theta$. If $q \leq
2$ then the inequality in \cite{F1} provides modulus of power type
$2/\theta$ and therefore cotype $2/\theta$.\\

It only remains to show that $\mathcal F$ has type
$[1-(\theta/2)]^{-1}$ in case (1) and $p$ in case (2). So pick $n$
vectors $x_1,\ldots,x_n$ in $\mathcal F$ and without loss of
generality assume that they are finitely supported and non-zero.
By a result of \cite{5a}, see Theorem 2 of \cite{F1}, we may find
for each $x_j$ an interpolation function $F_j$  such that
$F_j(\theta)=x_j$, and such that almost everywhere in $t$,
$$\|F_j(it)\|_t=\|x_j\|\ {\rm\ \  and\ \ \ } \|F_j(1+it)\|_q=\|x_j\|.$$
Fixing $\lambda_j>0$ for each $j$, we define
$$G_j(z)=\lambda_j^{z-\theta}F_j(z),$$
and observe that $G_j(\theta)=x_j$ and that almost everywhere in
$t$,
$$\|G_j(it)\|_t=\lambda_j^{-\theta}\|x_j\|\ {\rm\ \  and\ \ \ } \|G_j(1+it)\|_q=\lambda_j^{1-\theta}\|x_j\|.$$
Let $\epsilon_j=\pm 1$ for each $j$. By the formula (2) and this
observation we have
$$\|\sum_j \epsilon_j x_j\|_F=\leq \Big(\int_R\|\sum_j \epsilon_j G_j(it)\|_t d\mu_0(t)\Big)^{1-\theta}
\Big(\int_R\|\sum_j \epsilon_j G_j(1+it)\|_q
d\mu_1(t)\Big)^{\theta}$$
$$\leq \Big(\sum_j \lambda_j^{-\theta}\|x_j\|\Big)^{1-\theta} \Big(\int_R\|\sum_j \epsilon_j G_j(1+it)\|_q d\mu_1(t)\Big)^{\theta}.$$
Therefore
$$\|\sum_j \epsilon_j x_j\|_F^{1/\theta} \leq
 \Big(\sum_j \lambda_j^{-\theta}\|x_j\|\Big)^{\frac{1-\theta}{\theta}} \Big(\int_R\|\sum_j \epsilon_j G_j(1+it)\|_q d\mu_1(t)\Big),$$
and
$$2^{-n}\sum_{\epsilon_j=\pm 1}\|\sum_j \epsilon_j x_j\|_F^{1/\theta} \leq
\Big(\sum_j
\lambda_j^{-\theta}\|x_j\|\Big)^{\frac{1-\theta}{\theta}}
\Big(\int_R 2^{-n} \sum_{\epsilon_j=\pm 1}\|\sum_j \epsilon_j
G_j(1+it)\|_q d\mu_1(t)\Big).$$ Now it is known that $\ell_q$ has
type $r=\min(2,q)$, therefore there is a constant $C_q$ such that
$$2^{-n}\sum_{\epsilon_j=\pm 1}\|\sum_j \epsilon_j x_j\|_F^{1/\theta} \leq C_q
\Big(\sum_j
\lambda_j^{-\theta}\|x_j\|\Big)^{\frac{1-\theta}{\theta}}
\Big(\int_R (\sum_j \|G_j(1+it)\|_q^r)^{1/r} d\mu_1(t)\Big)$$
$$\leq C_q \Big(\sum_j \lambda_j^{-\theta}\|x_j\|\Big)^{\frac{1-\theta}{\theta}}
\Big(\sum_j \lambda_j^{(1-\theta)r} \|x_j\|^r\Big)^{1/r}.$$
Picking each $\lambda_j$ of the form $\|x_j\|^\alpha$, $\alpha \in
\R$,
$$\Big(2^{-n}\sum_{\epsilon_j=\pm 1}\|\sum_j \epsilon_j x_j\|_F^{1/\theta}\Big)^{\theta} \leq C_q^{\theta}
\Big(\sum_j \|x_j\|^{1-\alpha\theta}\Big)^{1-\theta} \Big(\sum_j
\|x_j\|^{r+\alpha(1-\theta)r}\Big)^{\theta/r}.$$ Choosing $\alpha$
such that
$$1-\alpha\theta=r+\alpha(1-\theta)r,$$ or equivalently
$$\alpha=\frac{1-r}{\theta+(1-\theta)r},$$
we obtain
$$\Big(2^{-n}\sum_{\epsilon_j=\pm 1}\|\sum_j \epsilon_j x_j\|_F^{1/\theta}\Big)^{\theta} \leq C_q^{\theta} \Big(\sum_j \|x_j\|^{1-\alpha\theta}\Big)^{1-\theta+\theta/r}.$$
Letting $1/m=1-\theta+\theta/r$, it is immediate by the choice of
$\alpha$ that $1-\alpha\theta=m$, and therefore
$$\Big(2^{-n}\sum_{\epsilon_j=\pm 1}\|\sum_j \epsilon_j x_j\|_F^{1/\theta}\Big)^{\theta} \leq C_q^{\theta} \Big(\sum_j \|x_j\|^{m}\Big)^{1/m}.$$
Since $1/\theta>1$, by \cite[Thm. 1.e.13]{lindtzaf} this is enough
to deduce that $\mathcal F$ has type $m$. Now if $q \leq 2$, then
$m=p$ and $\mathcal F$ has type $p$; if $q \geq 2$ then
$1/m=1-\theta/2$ and $\mathcal F$ has type $[1-(\theta/2)]^{-1}$.
This concludes the proof of the proposition.\hfill $\square$\\


We conclude with the proof of Lemma \ref{lemma}.

\noindent \emph{Proof of Lemma \ref{lemma}} It is quite similar to the proof
of Lemma 3.3 of \cite{AS}, up to some change and simplification of
notation. By the definition of $\|\cdot\|_t$ in \cite{F1}, a
norming subset $B_t$ of the unit ball of $X_t^*$ is obtained by
the following inductive procedure. Let
$$D_1=\{\lambda_n e_n, n \in \N, |\lambda| \leq 1\}.$$ Given $D_{n-1}$ a subset of $c_{00}$,
 let $D_n^1$ be the set of functionals of the form $$z^*=E\sum_{i=1}^l \alpha_i z_i^*,$$ where
$\sum_{i=1}^l |\alpha_i| \leq 1$, $z_i^* \in D_{n-1}$ and $E$ is
an interval. Let $D_n^2$ be the set of functionals
  of the form
$$z^*=E(\frac{1}{f(l)}\sum_{i=1}^l  z_i^*),$$
where $z_i^* \in D_{n-1}$, $z_1^*<\cdots<z_n^*$, and $E$ is an
interval. Let $D_n^3$ be the set of functionals of the form
$$z^*=EG(it),$$
where $E$ is an interval of integers and $G$ a special analytic
function, therefore
$$G=\frac{1}{\sqrt{f(k)}^{1-z}k^{z-z/q}}(\sum_{i=1}^k  G_i),\ {\rm with\ } G_i=\frac{1}{f(m_i)^{1-z}m_i^{z-z/q}}(\sum_{j=1}^{m_i}G_{i,j}),$$
where $m_1=j_{2k}$ and $m_{j+1}=\sigma(G_1,\ldots,G_j)$.

 Then let $$D_n=D_n^1 \cup D_n^2 \cup D_n^3$$ and let $$B_t=\cup_{n=0}^{\infty}D_n.$$

The result stated in the lemma will be proved for $z^* \in D_n$ by
induction on $n$. For $n=0$, that is, $z^*=\lambda e_i^*$ where
$|\lambda|=1$, we have that $J(z^*)=\emptyset$,  and we just
define $T_0(z^*)=z^*$. Now assuming the conclusion is proved for
any functional in $D_n$, we need to prove it for any $z^*$ in
$D_{n+1}^1, D_{n+1}^2$ or $D_{n+1}^3$.\\

When $z^*=0$ we have $T_0(z^*)=0$ and define $T_j(z^*)=0$ for all
$j \in J$. Although $J(0)=\emptyset$, and therefore $T_j(0)$ does
not appear in the formula of statement of the lemma, notation will
be simplified by giving a value to any $T_j(0)$. We now turn our
attention to $z^* \neq 0$.\\

If $z^* \in D_{n+1}^1$, then $z^*$ has the form $E(\sum_{i=1}^l
\alpha_i z_i^*)$, where $z_i^* \in D_n$, $\sum_{i=1}^l |\alpha_i|
\leq 1$ and $E={\rm ran\ } z^*$. Then we may apply the formula of
\cite{AS},  Lemma 3.3, Case 1, using as they do the fact that
$\cup_{i=1}^l J(E z_i^*) \subset J(z^*)$. That is,
$$T_0(z^*)=\sum_{i=1}^l \alpha_i T_0(Ez_i^*),$$
and
$$T_j(z^*)=\sum_{1 \leq i \leq l, j \in J(Ez_i^*)} \alpha_i T_j(Ez_i^*),$$
for $j \in J(z^*)$ (this sum being possibly $0$ if $j$ belongs to
no $J(Ez_i^*)$).\\

If $z^* \in D_{n+1}^2$, that is $z^*=E(\frac{1}{f(l)}\sum_{i=1}^l
z_i^*)$, where the $z_i^*$ are successive in $D_n$, then we
observe that once again $\cup_{i=1}^l J(Ez_i^*) \subset J(z^*)$,
and also, by the injectivity of $\sigma$,  that $J(Ez_i^*) \cap
J(Ez_s^*) = \emptyset$ whenever $i \neq s$. We  therefore may
apply the formula of \cite{AS}, Lemma 3.3, Case 2:
$$T_0(z^*)=\frac{1}{f(l)}\sum_{i=1}^l T_0(Ez_i^*),$$
and
$$T_j(z^*)=\frac{1}{f(l)}T_j(Ez_i^*)$$
when $j$ belongs to some $J(Ez_i^*)$, or $$T_j(z^*)=0$$
otherwise.\\

Finally, if $z^* \in D_{n+1}^3$, then
$$z^*=E(\frac{1}{\sqrt{f(k)}}\sum_{i=1}^k  z_i^*),\ {\rm with\ }
z_i^*=\frac{1}{f(m_i)}\sum_{j=1}^{m_i}z_{i,j}^*,,$$ where
$m_1=j_{2k}$ and $m_{j+1}=\sigma(G_1,\ldots,G_j)$ for
$G_1,\ldots,G_l$ associated to $z_1^*,\ldots,z_l^*$ by
$z_j^*=G_j(it)$. Let
$$i_1=\min\{i \in \{1,\ldots,l\}: E \cap {\rm supp\ }(z_i^*) \neq \emptyset\}.$$
By the induction hypothesis, we have
$$Ez^*=\frac{1}{\sqrt{f(l)}}\Big(\frac{1}{f(m_{i_1})}\sum_{j=1}^{m_{i_1}}Ez_{i_1,j}^*+
\sum_{i=i_1+1}^l \frac{1}{f(m_i)}\sum_{j=1}^{m_i}Ez_{i,j}^*\Big)$$
$$=\frac{1}{\sqrt{f(l)}}\frac{1}{f(m_{i_1})}\sum_{j=1}^{m_{i_1}}T_0(Ez_{i_1,j}^*)$$
$$+\sum_{i=i_1+1}^l  \frac{1}{\sqrt{f(l)}}\frac{1}{f(m_i)} \sum_{j=1}^{m_i} T_0(Ez^*_{i,j})$$
$$+\sum_{i=i_1}^l \sum_{j=1}^{m_i} \Big(\sum_{k \in J(Ez_{i,j}^*)} \frac{1}{\sqrt{f(l)}}\frac{1}{f(m_i)}T_k(Ez_{i,j}^*)\Big).$$
We then set
$$T_0(Ez^*)=\frac{1}{\sqrt{f(l)}}\frac{1}{f(m_{i_1})}\sum_{j=1}^{m_{i_1}}T_0(Ez_{i_1,j}^*)$$
and after noting that, by injectivity of $\sigma$,
$\{m_{i_1+1},\ldots,m_l\}$ and $J(Ez_{i,j}^*)$, $i=i_1,\ldots,l$,
$j=1,\ldots,m_i$ are mutually disjoint subsets of $J(Ez^*)$
(possibly empty when $Ez_{i,j}^*=0$), we set
$$T_k(Ez^*)=\frac{1}{\sqrt{f(l)}}\frac{1}{f(m_i)} \sum_{j=1}^{m_i} T_0(Ez^*_{i,j}),$$
if $k=m_i$ for some $i$ in $\{i_1+1,\ldots,l\}$,
$$T_k(Ez^*)=\frac{1}{\sqrt{f(l)}}\frac{1}{f(m_i)}T_k(Ez_{i,j}^*),$$
if $k \in J(z_{i,j}^*)$ for some $i \in \{i_1, \ldots, l\}$ and $j
\in \{1,\ldots,m_i\}$, and
$$T_k(Ez^*)=0$$ if $k \in J(Ez^*)$ otherwise. It is then easy to see that the conclusion of the lemma is
satisfied. \hfill $\square$

\author{Jes\'us M. F. Castillo, Departamento de
Matem\'aticas, Universidad
 de Extremadura, Avda de Elvas s/n, 06011
Badajoz, Espa\={n}a. \email{castillo@unex.es}}\\

\author{Valentin Ferenczi, Departamento de Matem\'atica, Instituto de Matem\'atica e
Estat\'istica, Universidade de S\~{a}o Paulo, rua do Mat\~{a}o,
1010, 05508-090 S\~{a}o Paulo, SP, Brazil and  Equipe d'Analyse
Fonctionnelle, Institut de Math\'ematiques, Universit\'e Pierre et
Marie Curie - Paris 6, Case 247, 4 place Jussieu, 75252 Paris
Cedex 05, France.  \email{ferenczi@ime.usp.br}}\\

\author{Yolanda Moreno, Departamento de Matem\'aticas,
Escuela Polit\'ecnica de C\'aceres, Universidad
 de Extremadura, Avda de la Universidad s/n, 07011
C\'aceres, Espa\={n}a. \email{ymoreno@unex.es}}

\end{document}